\documentclass[reqno]{amsart}

\usepackage{amsfonts}
\usepackage[all]{xy}
\usepackage{amssymb}
\usepackage{amsmath}
\usepackage{epsfig}
\usepackage{amscd}
\usepackage{graphicx, color}
\usepackage{epstopdf}
\usepackage{amscd}
\usepackage{amssymb}
\usepackage{amstext}
\usepackage[all]{xy}
\usepackage{mathrsfs}

\def\E{\ifmmode{\mathbb E}\else{$\mathbb E$}\fi} %natural numbers
\def\N{\ifmmode{\mathbb N}\else{$\mathbb N$}\fi} %natural numbers%
\def\R{\ifmmode{\mathbb R}\else{$\mathbb R$}\fi} %real numbers
\def\Q{\ifmmode{\mathbb Q}\else{$\mathbb Q$}\fi} %rational numbers
\def\C{\ifmmode{\mathbb C}\else{$\mathbb C$}\fi} %complex numbers
\def\H{\ifmmode{\mathbb H}\else{$\mathbb H$}\fi} %complex numbers
\def\Z{\ifmmode{\mathbb Z}\else{$\mathbb Z$}\fi} %integers
\def\P{\ifmmode{\mathbb P}\else{$\mathbb P$}\fi} %real numbers
\def\T{\ifmmode{\mathbb T}\else{$\mathbb T$}\fi} %real numbers
\def\SS{\ifmmode{\mathbb S}\else{$\mathbb S$}\fi} %real numbers
\def\DD{\ifmmode{\mathbb D}\else{$\mathbb D$}\fi} %real numbers

\renewcommand{\i}{\iota}

\newcommand{\del}{\partial}

\newcommand{\ben}{\begin{enumerate}}
\newcommand{\een}{\end{enumerate}}
\newcommand{\be}{\begin{equation}}
\newcommand{\ee}{\end{equation}}
\newcommand{\bea}{\begin{eqnarray}}
\newcommand{\eea}{\end{eqnarray}}
\newcommand{\beastar}{\begin{eqnarray*}}
\newcommand{\eeastar}{\end{eqnarray*}}
\newcommand{\bc}{\begin{center}}
\newcommand{\ec}{\end{center}}

\newtheorem{thm}{Theorem}[section]
\newtheorem{cor}[thm]{Corollary}
\newtheorem{lem}[thm]{Lemma}
\newtheorem{prop}[thm]{Proposition}

\theoremstyle{definition}
\newtheorem{defn}[thm]{Definition}
\newtheorem{rem}[thm]{Remark}
\newtheorem{cond}[thm]{Condition}

\newtheorem{ques}[thm]{Question}

\newtheorem*{thm*}{Theorem}

\numberwithin{equation}{section}

\def\R{{\mathbb R}}

\def\E{{\mathbb E}}
\def\Z{{\mathbb Z}}
\def\C{{\mathbb C}}
\def\R{{\mathbb R}}
\def\P{{\mathbb P}}

\def\N{{\mathbb N}}

\def\11{{\mathbb I}}

\def\H{\mathbb{H}}

\def\C{\mathbb{C}}
\def\Z{\mathbb{Z}}

\def\T{\mathbb{T}}

\def\Q{\mathbb{Q}}

\def\E{\ifmmode{\mathbb E}\else{$\mathbb E$}\fi} %natural numbers
\def\N{\ifmmode{\mathbb N}\else{$\mathbb N$}\fi} %natural numbers
\def\R{\ifmmode{\mathbb R}\else{$\mathbb R$}\fi} %real numbers
\def\Q{\ifmmode{\mathbb Q}\else{$\mathbb Q$}\fi} %rational numbers
\def\C{\ifmmode{\mathbb C}\else{$\mathbb C$}\fi} %complex numbers
\def\Z{\ifmmode{\mathbb Z}\else{$\mathbb Z$}\fi} %integers
\def\P{\ifmmode{\mathbb P}\else{$\mathbb P$}\fi} %real numbers
\def\CS{\ifmmode{\mathbb S}\else{$\mathbb S$}\fi} %real numbers
\def\DD{\ifmmode{\mathbb D}\else{$\mathbb D$}\fi} %real numbers

\def\R{{\mathbb R}}

\def\E{{\mathbb E}}
\def\Z{{\mathbb Z}}
\def\C{{\mathbb C}}
\def\R{{\mathbb R}}

\def\N{{\mathbb N}}
\def\CA{{\mathcal A}}

\def\CG{{\mathcal G}}
\def\CH{{\mathcal H}}

\def\CL{{\mathcal L}}

\def\CS{{\mathcal S}}

\def\CX{{\mathcal X}}

\def\darr#1{\raise1.5ex\hbox{$\leftrightarrow$}
\mkern-16.5mu #1}

\def\roughly#1{\raise.3ex\hbox{$#1$\kern-.75em
\lower1ex\hbox{$\sim$}}}

\def\opname#1{\mathop{\kern0pt{\rm #1}}\nolimits}

\def\Im{\opname{Im}}

\def\vol{\opname{vol}}

\def\dudt{\frac{\partial u}{\partial t}}

\def\supp{\operatorname{supp}}

\def\Spec{\operatorname{Spec}}

\usepackage{xcolor}
\usepackage{hyperref}

\hypersetup{
    colorlinks,
    linkcolor={red!90!black},
    citecolor={blue!90!black},
    urlcolor={blue!90!black},
    linktocpage
}

\begin{document}

\title[Lagrangian spectral invariants]{
Exact Lagrangian submanifolds, Lagrangian spectral invariants and Aubry--Mather theory}

\author{Lino Amorim}
\address{Department of Mathematics and Statistics, Boston University, USA}
\email{lamorim@bu.edu}

\author{Yong-Geun Oh}
\address{Center for Geometry and Physics, Institute for Basic Sciences (IBS),
Pohang 37673, Korea \&
Department of Mathematics, POSTECH, Pohang 37673, KOREA}
\email{yongoh1@postech.ac.kr}

\author{Joana Oliveira dos Santos}
\address{Department of Mathematics and Statistics, Boston University, USA}
\email{jamorim@bu.edu}

\thanks{The first-named author was supported by EPSRC grant EP/J016950/1. The second-named author is supported by the IBS project \# IBS-R003-D1.}

%\date{}

\begin{abstract}
We construct graph selectors for compact exact Lagrangians in the cotangent bundle of an orientable, closed manifold. The construction combines Lagrangian spectral invariants developed by Oh and results by Abouzaid about the Fukaya category of a cotangent bundle. We also introduce the notion of Lipschitz-exact Lagrangians and prove that these admit an appropriate generalization of graph selector. We then, following Bernard--Oliveira dos Santos, use these results to give a new characterization of the Aubry and Ma\~n\'e sets of a Tonelli Hamiltonian and to generalize a result of Arnaud on Lagrangians invariant under the flow of such Hamiltonians.
\end{abstract}


\maketitle

\medskip

\tableofcontents

\section{Introduction}

Consider the cotangent bundle $T^*M$ of a closed orientable $n$-manifold $M$. Let $\theta$ be the canonical Liouville one-form in $T^*M$. A Lagrangian embedding $\iota: N \to T^*M$ is called
exact if $\iota^*\theta$ is an exact one-form, that is, $\iota^*\theta = dS$ for some function $S:N \to \R$.
We call any such function a \emph{Liouville primitive} of the exact Lagrangian embedding
$\iota: N \to T^*M$. We will denote
$$
L = \iota(N)
$$
as a submanifold of $T^*M$. In the present paper, we will distinguish the domain of an embedding $\iota$ and
its image by different letters by exclusively denoting
$
L = \iota(N)
$
while $N$ denotes an abstract $n$-manifold. We also denote by $i:L \hookrightarrow T^*M$
the inclusion map and by $h_L$ the function
$h_L: L \to \R$ such that
\be\label{eq:h}
h_L(y) = S \circ \iota^{-1}(y), \quad y \in L.
\ee
We will often omit the sub-index $L$ and denote $h = h_L$.
This is nothing but the
Liouville primitive for the inclusion $i:L \to T^*M$, that is  $i^*\theta = dh_L$.

In many problems in the symplectic topology and Hamiltonian dynamics (see \cite{PPS} and \cite{bernard-santos}, for example) it is crucial to construct \emph{graph selectors} for compact, exact Lagrangian submanifolds of $T^*M$.

\begin{defn}
	Let $\iota: N \to T^*M$ be a compact, exact Lagrangian embedding with Liouville primitive $S$. A \emph{graph selector} for $L$ is a Lipschitz function
	$f:M\to\mathbb{R}$ such that $f$ is differentiable on a dense open set $\mathcal{U} \subset M$ of full measure
	and for all points $q \in \mathcal{U}$ we have
	$$
	(q,df(q))\in L \ \textrm{ and }\ f(q)= h_L(q,df(q)).
	$$	
\end{defn}
The second equation in this definition is equivalent to
\be\label{eq:fh}
f\circ \pi|_L = h_L
\ee
on $L$, where $\pi|_L$ is the restriction to $L$ of the projection $\pi:T^*M \to M$.

There are several proofs that graph selectors exist for Lagrangians that are Hamiltonian isotopic to the zero-section. The proofs by Chaperon \cite{chaperon}, Paternain--Polterovich--Siburg \cite{PPS} and Viterbo \cite{viterbo:generating} use \emph{generating functions} in the sense of classical mechanics or  micro-local analysis \cite{hormander}.
A more global definition of generating function of an exact Lagrangian
submanifold is given in \cite{sikorav:gen,viterbo:generating}, and it is proved by
Laudenbach--Sikorav \cite{laud-sikorav} that any Lagrangian Hamiltonian isotopic to the zero-section admits one.
However a general compact exact Lagrangian submanifold
is not known to have a generating function. Therefore this method of generating functions
cannot be applied to general exact Lagrangian submanifolds (at least for now).

There is a second method to construct graph selectors developed by the second-named author
in \cite{oh:jdg} (and in other later literature), that uses Floer theory instead. This relies on being able to compute the Floer homology of $L$ with a cotangent fiber. Then one exploits the filtration
present in the Floer complex to construct a spectral invariant which defines the function $f$, this is called the \emph{basic phase function}. One of the main purposes of the present paper is to extend these constructions to general exact Lagrangian submanifolds and to explore their applications to Aubry--Mather theory.

One of the outstanding problems in symplectic geometry is the \emph{Arnold nearby Lagrangian conjecture} which states that any compact exact Lagrangian is in fact Hamiltonian isotopic to the zero-section. Therefore a proof of this conjecture would immediately imply the existence of graph-selectors for any compact exact Lagrangian. This conjecture seems to be completely out of reach at the moment. However there have been spectacular advances in the understanding of this problem in the past few years, see \cite{abouzaid2, FSS, nadler, kra}. For example, it is now known that the projection from a compact exact Lagrangian to the zero-section is a homotopy equivalence.

For our purposes the most important  results are results by  Abouzaid \cite{abouzaid2} (using an argument due to \cite{FSS}) and Abouzaid--Kragh \cite{kra} that prove the Floer homology between any compact, exact Lagrangian submanifold $L$ and a fiber $T_q^*M$ is isomorphic to $\Z$.
This enables us to imitate constructions
given in \cite{oh:jdg,kasturi-oh1} and prove the following

\begin{thm}\label{gs}
	Let $L \subset T^*M$ be a closed smooth
	exact embedded Lagrangian submanifold. Then it admits a \emph{graph selector}.
\end{thm}

In the last couple of decades since
the appearance of Eliashberg--Gromov's symplectic $C^0$-rigidity theorem \cite{eliash},
symplectic topology and Hamiltonian dynamics have witnessed many interesting
$C^0$ phenomena. Some attempts to organize these phenomena in some conceptual
way have been explored. For example, the second named author and M\"uller
introduced the notion of $C^0$-Hamiltonian diffeomorphisms, called \emph{Hamiltonian homeomorphism} (abbreviated as \emph{hameomorphism})
in \cite{oh:hameo1}, which has been further studied in later literature
such as \cite{mueller}, \cite{HLS}.

In this paper we give a modest step in the direction of further understanding this $C^0$ nature
of the symplectic world: we define \emph{Lipschitz-exact Lagrangians}. This definition is
partially motivated by the constructions in Aubry--Mather theory \cite{bernard-santos}, where the
study of such a class of Lagrangians is needed naturally.
Imitating the construction of Hamiltonian homeomorphisms in \cite{oh:hameo1}, we
introduce them in Definition \ref{defn:Lip-exact}, by taking a completion of smooth exact Lagrangians in a suitably chosen topology. Roughly, a Lipschitz-exact Lagrangian consists of a smooth manifold $N$, a Lipschitz function $S$ on $N$ and a Lipschitz embedding
$$\iota: N\longrightarrow T^*M,$$
that can be approximated by a sequence of smooth exact Lagrangians $\iota_k:N\to T^*M$.

The main result we prove about Lipschitz-exact Lagrangians is Theorem \ref{gen gs}, which states that any compact Lipschitz-exact Lagrangian admits a \emph{generalized graph selector} (Definition \ref{def:gen_slector}). This is a weakening of the notion of graph selector, which was introduced in \cite{bernard-santos} and plays a crucial role in our results on Hamiltonian dynamics.

We will give two applications of our results guaranteeing the existence of graph selectors, following the ideas of \cite{bernard-santos1} and \cite{bernard-santos}. The first is in Aubry--Mather theory. Here we are interested in studying the dynamics of
$C^2$ \emph{Tonelli} Hamiltonians $H:T^*M\to \mathbb{R}$, which means convex with positive definite Hessian and superlinear in each fiber. These Hamiltonian are sometimes called \emph{optical}.

We first introduce some notation we need to state the result. Denote by $\mathcal{G}$ the set
$$\mathcal{G}:=\{\Gamma_v\mid v\in{C^{1,1}(M)}\},$$ where
$
\Gamma_v := \{(x,dv(x)) \in T^*M \mid x \in M\}$ and $C^{1,1}(M)$ are $C^1$ functions with Lipschitz differential.
Let $Ham(T^*M)$ be the set of Hamiltonian diffeomorphisms, $Symp^{e}(T^*M)$ be the set of exact symplectic diffeomorphisms
and  define the sets
$$\mathcal{H}:=\{\varphi(\Gamma_v)\mid \varphi\in Ham(T^*M), \Gamma_v\in\mathcal{G}\}$$ and
$$\mathcal{E}:=\{\varphi(\Gamma_v)\mid \varphi\in Symp^{e}(T^*M), \Gamma_v\in\mathcal{G}\}.$$
We define one more set,
$$\mathcal{L}:=\{\textrm{compact Lipschitz-exact Lagrangians in } T^*M\},$$
and remark that $\mathcal{H}\subset\mathcal{E}\subset\mathcal{L}$. The second inclusion is proved in Lemma \ref{graph potential}.
(We remark that in \cite{bernard-santos} the current $\CH$ is denoted by $\CL$.
Here we reserve $\CL$ for the set of \emph{Lipschitz-exact Lagrangians} instead, which are
defined in the present paper.)

Fix a $C^2$ Tonelli Hamiltonian $H:T^*M\to \mathbb{R}$ and an energy value $a\in\mathbb{R}$, for each manifold $L$ in $\mathcal{G}$, $\mathcal{H}$, $\mathcal{E}$ or $\mathcal{L}$ (or any other set of compact submanifolds of $T^*M$) we define the maximal invariant subset of $L\cap \{H=a\}$ by
$$\mathcal{I}_a^*(L):=\bigcap_{t\in\mathbb{R}}\phi_H^t\left(L\cap \{H=a\}\right),$$ where $\phi_H^t$ is the Hamiltonian flow of $H$. We define the \emph{critical value} of $H$
\begin{align}
\alpha_\mathcal{E}(H):=\inf_{L\in\mathcal{E}}\max_{(q,p)\in L}H(q,p).\label{defalpha}
\end{align}
We now define the Aubry set
$\mathcal{A}_\mathcal{E}^*(H)$ and the Ma\~n\'e set $\mathcal{N}_\mathcal{E}^*(H)$ as
\begin{align}\label{defAN}
\mathcal{A}_
\mathcal{E}^*(H)& :=  \bigcap_{L\in\mathcal{E}, L\subset\{H\leq\alpha_\mathcal{E}(H)\}}\mathcal{I}^*_{\alpha_\mathcal{E}(H)}(L),\nonumber\\
\mathcal{N}_\mathcal{E}^*(H)& :=  \bigcup_{L\in\mathcal{E}, L\subset\{H\leq\alpha_\mathcal{E}(H)\}}\mathcal{I}^*_{\alpha_\mathcal{E}(H)}(L).
\end{align}
We can give analogous definitions, $\alpha_\mathcal{S}(H)$, $\mathcal{A}_\mathcal{S}^*(H)$ and $\mathcal{N}_\mathcal{S}^*(H)$, by taking $L\in\mathcal{S}$, where $\mathcal{S}$ can be $\mathcal{G}$, $\mathcal{H}$ or $\mathcal{L}$. The classical definitions of these objects use $\mathcal{G}$. We would like to remark that it follows from \cite{bernard2} that $\alpha_\mathcal{G}(H)$ is in fact a minimum.

In \cite{bernard-santos1} and \cite{bernard-santos} it is shown that if we use $\mathcal{H}$ we obtain the same objects and it is asked whether the same kind of result hold for the objects in $\mathcal{E}$.

Here, using the graph selector constructed in Theorem \ref{gs}, we show that if we use $\mathcal{E}$ or even $\mathcal{L}$ we again obtain the same objects.
\begin{thm}\label{aplic1}
	Let $H:T^*M\to \mathbb{R}$ be a Tonelli Hamiltonian. Then
	$$\alpha_\mathcal{G}(H)=\alpha_\mathcal{E}(H)=\alpha_\mathcal{L}(H),$$
	$$\mathcal{A}_\mathcal{G}^*(H)=\mathcal{A}_\mathcal{E}^*(H)=\mathcal{A}_\mathcal{L}^*(H),$$
	$$\mathcal{N}_\mathcal{G}^*(H)=\mathcal{N}_\mathcal{E}^*(H)=\mathcal{N}_\mathcal{L}^*(H).$$
\end{thm}

The reason for taking this approach is that definitions (\ref{defalpha}) and (\ref{defAN}) are symplectically natural, meaning that the following result obtained by Bernard \cite{bernard1} using variational methods becomes trivial:

\begin{cor}
	Let $H:T^*M\to \mathbb{R}$ be a Hamiltonian (not necessarily Tonelli). Then definitions (\ref{defalpha}) and (\ref{defAN}) are equivariant under the action of $Symp^{e}(T^*M)$. In other words given $\varphi\in Symp^{e}(T^*M)$ we have
	$$\alpha_\mathcal{E}(H\circ\varphi)=\alpha_\mathcal{E}(H),\ \varphi\left(\mathcal{A}_\mathcal{E}^*(H\circ \varphi)\right)=\mathcal{A}_\mathcal{E}^*(H),\ \varphi\left(\mathcal{N}_\mathcal{E}^*(H\circ \varphi)\right)=\mathcal{N}_\mathcal{E}^*(H).$$
\end{cor}
\begin{proof}
	This easily follows from the definitions and the fact that if $L\in \mathcal{E}$ and $\varphi\in Symp^{e}(T^*M)$, then $\varphi(L)\in \mathcal{E}$.
\end{proof}

We observe that this result is only useful for the classical (using $\mathcal{G}$) Aubry and Ma\~n\'e sets when both $H$ and $H\circ \varphi$ are Tonelli, that is, when Theorem 1.3 applies.

The next theorem is our second application.

\begin{thm}\label{aplic2}
	Suppose M is connected and let $(N,\iota,S)$ be a compact Lipschitz-exact Lagrangian in $T^*M$ in the sense of Definition \ref{defn:Lip-exact}.
	If $L = \iota(N)$ is invariant under the flow of a Tonelli Hamiltonian then $L \in\mathcal{G}$, that is, $L$ is a Lipschitz graph.
\end{thm}

A version of this theorem for smooth Lagrangians Hamiltonian isotopic to the zero section was proved by Arnaud \cite{arnaud}. It was then generalized in \cite{bernard-santos} to Lagrangians in $\mathcal{H}$.
%
%\begin{thm}\label{aplic2}
%Suppose M is connected and let $L$ be a Lipschitz-exact Lagrangian in $\mathcal{E}$.
%If $L$ is invariant under the flow of a Tonelli Hamiltonian then $L \in\mathcal{G}$, that is, $L$ is a Lipschitz graph.
%\end{thm}
%
%A version of this theorem for smooth Lagrangians Hamiltonian isotopic to the zero section was proved by Arnaud \cite{arnaud}. It was then generalized in \cite{bernard-santos} to Lagrangians in $\mathcal{H}$. We expect this theorem to hold for all compact Lipschitz-exact Lagrangians in the sense of Definition \ref{defn:Lip-exact}. For this we would need to extend Lemma \ref{inv} to this larger class. See Remark \ref{c0lag} for further discussion on this problem.

\section{Lipschitz-exact Lagrangian submanifolds}
\label{sec:exact}

We start by recalling the notion of smooth exact Lagrangian submanifolds. We restrict ourselves to cotangent bundles, since that is the case we are interested in, but most of these notions make sense in general exact symplectic manifolds.
Let $M$ be a closed smooth manifold of dimension $n$. We consider its cotangent bundle $X=T^*M$ with the usual symplectic form $\omega=-d \theta$, where $\theta$ is the canonical Liouville $1$-form.

\begin{defn}
	Let $N$ be a smooth $n$-manifold and $\iota: N \to T^*M$ be a smooth embedding. We say $(N,\iota)$  is an \emph{exact Lagrangian embedding} if $\iota^*\theta$ is exact, and call its image $L = \iota(N)$ an exact Lagrangian submanifold of $T^*M$.
	We call the triple $(N,\iota,S)$ a \emph{Lagrangian brane} where $S$ is
	a smooth function $S:N \to \R$ such that $ d S=\iota^*\theta$. We call $S$
	a \emph{Liouville primitive} or just a \emph{primitive} of $(N,\iota)$.
\end{defn}

\begin{defn} A diffeomorphism  $\varphi: T^*M \to T^*M$ is called an \emph{exact symplectomorphism} if $\varphi^*\theta -\theta$ is an exact $1$-form. We denote by $Symp^{e}(T^*M)$ the set of all exact symplectomorphisms.
\end{defn}

In this paper we will be particularly interested in exact Lagrangians of the following form. Let $v:M\to \R$ be a smooth function and let
$$\Gamma_v:M\longrightarrow T^*M,\ q \mapsto (q, d v(q))$$
be the graph of its derivative. We have the following lemma

\begin{lem}\label{exact_primitive}
	Let $v$ be a smooth function and $\varphi\in Symp^{e}(T^*M)$. Then the embedding
	$$\iota=\varphi\circ \Gamma_v:M\longrightarrow T^*M$$
	is an exact Lagrangian.
\end{lem}
\begin{proof}
	Since $\varphi$ is exact, there is $g:T^*M\to\R$ such that $\varphi^*\theta -\theta=d g$. We compute
	\begin{align*}
	\iota^*\theta=& \Gamma_v^*(\varphi^*\theta)= \Gamma_v^*(\theta+d g)= \Gamma_v^*(\theta)+d(\Gamma_v^*(g)).
	\end{align*}
	By definition of $\theta$ we can easily see that $\Gamma_v^*(\theta)=d v$. Hence $\iota^*\theta=d(v+\Gamma_v^*g)$.
\end{proof}

We now introduce a generalization of the notion of exact Lagrangian that we call \emph{Lipschitz-exact Lagrangians}.

\begin{defn}\label{defn:Lip-exact}
	Let $N$ be a smooth $n$-manifold, $\iota: N \to T^*M$ be an injective continuous map
	and let $S:N \to\R$ be a continuous function. We say the triple $(N,\iota, S)$ is
	a \emph{Lipschitz-exact Lagrangian brane} if there are equi-Lipschitz sequences $S_k: N \to \R$ of smooth functions and $\iota_k: N \to T^*M$ of smooth embeddings such that
	\begin{enumerate}
		\item[\textbf{(a)}] $\iota_k^*\theta =d S_k$,
		\item[\textbf{(b)}] $S_k\to S$ and $\iota_k\to \iota$ uniformly (in the $C^0$ topology).
	\end{enumerate}
	We call any such sequence $(N, \iota_k, S_k)$ an approximating sequence of the Lipschitz-exact Lagrangian brane $(N,\iota, S)$.
\end{defn}

Note that as consequence of the definition, if $(N,\iota, S)$ is a Lipschitz-exact Lagrangian brane, then $\iota$ and $S$ are Lipschitz.
With this definition we can also see that $S$ is a primitive of the pull-back of $\theta$ in the sense of the next proposition.

\begin{prop}\label{prop:iota*omega=0} Let $(N,\iota,S)$ be a Lipschitz-exact Lagrangian brane. Then
	\be\label{eq:iotatheta=dS}
	\iota^*\theta = dS
	\ee
	as a one-form in $L^\infty$. In particular $\iota^*\omega = 0$ as a two-form in $L^\infty$.
\end{prop}
\begin{proof}
	We assume $N$ is orientable. (Otherwise we just use the odd differential forms
	in the sense of de Rham \cite{deRham} for the arguments below and so we will focus on the case of orientable $N$.)
	
	We fix a Riemannian metric on $N$ and denote by $\vol$ the associated volume form.
	We first note that both sides of the equation define well-defined currents since $\iota$ and $S$ are
	Lipschitz functions and so are in $W^{1,\infty}$ (see \cite[Section 4.2.3, Theorem 4]{evans}). By a standard theorem (see \cite[Section 6.2, Theorem 1]{evans}, for example),
	$dS$ equals its weak derivative almost everywhere. Therefore it is enough to prove that
	the weak derivative (or equivalently the derivative as a current) is the same as $\iota^*\theta$.
	
	Let $\eta$ be any smooth $(n-1)$-form.
	Then the weak derivative $dS$ satisfies
	\bea\label{weakdS}
	\int_N dS \wedge \eta & = & -\int_N S d \eta = -\lim_{k \to \infty} \int_N S_k  d\eta \\ \nonumber
	& = & \lim_{k \to \infty}\int_N dS_k \wedge \eta  = \lim_{k \to \infty} \int_N \i_k^*\theta \wedge \eta.
	\eea
	Here we use the uniform convergence of $S_k$ to $S$ for the second equality.
	Without loss of generality, we assume that $\supp \eta$ is contained in a Darboux neighbohood $U$
	equipped with canonical coordinates $(q,p)$, $q = (q_1,\ldots,q_n), \, p= (p_1,\ldots, p_n)$
	for which $\theta =p dq =\sum_{i=1}^n p_i  dq_i$.
	We now express $\iota_k^*\theta = p \circ \iota_k \cdot d(q \circ \iota_k)$ and rewrite
	$$
	\int_N \i_k^*\theta \wedge \eta = \int_U p\circ \iota_k \cdot d(q\circ \iota_k) \wedge \eta
	$$
	We will show
	$$
	\int_U p\circ \iota_k \cdot d(q\circ \iota_k) \wedge \eta
	\to \int_U p\circ \iota \cdot d(q\circ \iota) \wedge \eta.
	$$
	This is where the hypothesis that $\iota_k$ is equi-Lipschitz enters in a crucial way: it
	first implies that $q \circ \iota_k, p\circ \iota_k$ are also equi-Lipschitz.
	We rewrite
	\beastar
	&{}& \int_U p\circ \iota_k\cdot d(q \circ \iota_k) \wedge \eta - \int_U  p\circ \iota \cdot d(q \circ \iota) \wedge \eta\\
	& = & \int_U (p\circ \iota_k - p\circ \iota)\cdot d(q \circ \iota_k) \wedge \eta
	+ \int_U p\circ \iota \cdot d(q \circ \iota_k - q\circ \iota )\wedge \eta.
	\eeastar
	For the first integral, we have the bound
	$$
	\left|\int_U  (p\circ \iota_k - p\circ \iota)\cdot d(q \circ \iota_k)\wedge \eta\right|
	\leq \|p\circ \iota_k - p\circ \iota\|_{L^\infty} \int_U |d(q \circ \iota_k)|\cdot |\eta| \, \vol.
	$$
	Since $q \circ \iota_k$ are equi-Lipschitz, $d(q \circ \iota_k)$ have an uniform bound on $L^\infty$-norm. 
	Therefore there exists $C > 0$
	independent of $k$ such that
	$$
	\int_U |d(q \circ \iota_k)|\cdot |\eta|\, \vol \leq C
	$$
	for all $k$. Then by the uniform convergence of $p\circ \iota_k \to p\circ \iota$, we have derived
	$$
	\int_U  (p\circ \iota_k - p\circ \iota)\cdot d(q \circ \iota_k)\wedge \eta  \to 0.
	$$
	Next, we consider the integral
	$$
	\int_U p\circ \iota \cdot d(q \circ \iota_k - q \circ \iota )\wedge \eta.
	$$
	Again using the uniform bound on the $W^{1,\infty}$-norm of $q\circ \iota_k$, we can choose a subsequence,
	still denoted by $q\circ \iota_k$, weakly-$\star$ converging to $q\circ \iota$ in $W^{1,\infty}$. Therefore
	$$
	\int_U p\circ \iota \cdot d (q \circ \iota_k -  q\circ \iota )\wedge \eta
	=  (-1)^{n-1}  \int_U \left(\eta \wedge (p \circ \iota) d(q \circ \iota_k)
	- \eta \wedge (p \circ \iota)d(q \circ \iota)\right) \to 0.
	$$
	This proves the convergence
	\be\label{convergence}
	\int_U p\circ \iota_k \cdot d(q \circ \iota_k) \wedge \eta \to \int_U p \circ \iota \cdot d(q\circ \iota) \wedge \eta
	= \int_N \iota^*\theta \wedge \eta
	\ee
	after taking a subsequence.
	
	Combining \ref{weakdS} and \ref{convergence}, we have proved
	$$
	\int_N dS \wedge \eta = \int_N \iota^*\theta \wedge \eta
	$$
	for all smooth $(n-1)$-form $\eta$, that is, the weak derivative of $S$ is given by
	$dS = \iota^*\theta$.

	In particular $d(\iota^*\theta) = 0$ as a current. On the other hand
	by the standard identity $d \varphi^* = \varphi^*d$ acting on the set of currents for
	the Lipschitz map $\varphi$ (see e.g., \cite[Section 4.4.1]{federer}),
	we obtain $\iota^*d\theta = d(\iota^*\theta)$ and hence
	$\iota^*\omega = -\iota^*(d\theta) = 0$ as a current.
	%
	%
	%To prove the second statement, we first taking the differential of $\iota^*\theta= dS$
	%as a current and derive
	%$$
	%d(\iota^*\theta) = 0
	%$$
	%as a current. Since $\iota_k^*(d\theta) = d(\iota_k^*\theta) \to d(\iota^*\theta)$ as a current,
	%we obtain $\iota^*(d\theta) = 0$ as a current which is a two-form contained in $L^\infty$.
	This finishes the proof.
\end{proof}

The following proposition is proved in the same way as Proposition \ref{prop:iota*omega=0}.
It is also proved in Proposition 2 in \cite{bernard-santos} in a slightly different setting.

\begin{prop}\label{cor:lip pot}
	Let $(N,\iota, S)$ be an Lipschitz-exact Lagrangian brane and
	let $c:[0,1]\to N$ be a Lipschitz curve, then
	\be\label{eq:lip_pot2}
	\int_{\iota\circ c} \theta = S\left(c(1)\right)-S\left(c(0)\right).
	\ee
\end{prop}
\begin{proof} We follow the proof above, with $S\circ c$ and $(\iota\circ c)^*\theta$ instead of $S$ and $\iota^*\theta$. All the arguments apply since $S\circ c$ and $\iota\circ c$ are still Lipschitz.
\end{proof}
%
%\begin{rem} Examination of the above proof shows that, for the purposes of this proposition,
%we could drop the condition that the sequence $S_k$ be equi-Lipschitz from the
%definition of Lipschitz-exact Lagrangian brane.
%\end{rem}

Our main example of Lipschitz-exact Lagrangian branes are the elements of $\mathcal{E}$.

\begin{lem}\label{graph potential}
	Let $v:M\to \R$ be a $C^{1,1}$ function and let $\varphi: T^*M\to T^*M$ be an exact symplectomorphism with $\varphi^*\theta-\theta= d g$. Then $\iota:M\to T^*M$ given by $\iota=\varphi\circ\Gamma_v$ is Lipschitz with Lipschitz primitive $S=v+\Gamma^*_v g$. That is, the triple $(M, \iota, S)$ is a Lipschitz-exact Lagrangian brane.
	In other words, $\mathcal{E} \subset \mathcal{L}$.
\end{lem}
\begin{proof}
	Let $v_k$ be a sequence of smooth functions converging to $v$ in $C^1(M)$ and bounded in $W^{2,\infty}(M)$. Take $\iota_k=\varphi\circ \Gamma_{v_k}$ and $S_k=v_k+\Gamma_{v_k}^*g$. Lemma \ref{exact_primitive} now implies that $\iota_k^*\theta_k=d S_k$. Moreover the $S_k$ are equi-Lipschitz and converge uniformly to $S$ and similarly the sequence $\iota_k$ is equi-Lipschitz and converges uniformly to  $\iota$.
\end{proof}

We now discuss possible alternative definitions of Lipschitz-exact Lagrangians. Instead of requiring that the sequences $\iota_k$ and $S_k$ be equi-Lipschitz, we could have required that both $\iota_k$ and $S_k$ converge to $\iota$ and $S$ in $C^{1,1}$ respectively. But this definition would be too strong and we wouldn't be able to prove Lemma \ref{graph potential} (since smooth functions are not dense in $C^{1,1}(M)$).

On the other direction we could have dropped the equi-Lipschitz condition. Then one would get a definition of Lagrangian analogous to those from $C^0$-symplectic geometry. We will not do this here since for our purposes, namely the existence of graph selectors, this seems too weak. Also, for the applications to Aubry--Mather theory our definition suffices. 

However, it would be interesting to study this more general definition. One could ask if, whenever $\iota$ is a smooth embedding, then $\iota:N\to T^*M$ is an exact Lagrangian in the usual sense. This would be a version of Eliashberg--Gromov $C^0$ rigidity for exact Lagrangians. We observe that Proposition \ref{prop:iota*omega=0} implies this is the case for our more restrictive definition. 
For this more general definition another interesting question would be the following.
\begin{ques}
	Let $\iota_k$ and $S_k$ be uniformly convergent sequences.
	Assume a normalization condition on the $S_k$, say $S_k(x_0) = 0$ for all $k$ at a given point $x_0 \in N$.
	Will the uniform limit of $\iota_k$ determine the limit of $S_k$?
\end{ques}
This question would be a counterpart of the fundamental uniqueness result in topological Hamiltonian
dynamics given in \cite{oh:hameo1,viterbo:unique,oh:gokova,buh-sey}. 
We observe that for our definition of Lipschitz-exact Lagrangian, Proposition \ref{prop:iota*omega=0} answers this question affirmatively.
We hope to come back to the study of these questions elsewhere.

\section{Action functional and Floer homology}
\subsection{Wrapped Floer homology}

In this subsection we briefly review some aspects of wrapped Floer
cohomology as in \cite{abouzaid2}. Our conventions for the action functional and gradings will differ from the ones in \cite{abouzaid2}, we will instead follow \cite{oh:jdg}, and naturally obtain a homology complex.
In this section and the next two, we will exclusively consider a Lagrangian submanifold $L$
as a subset $T^*M$ because all constructions will depend only on the image.

For a given  $(N,\iota, S)$, we denote
$$
L = \iota(N), \quad h = S\circ \iota^{-1}.
$$
Then by definition, we have $i^*\theta = dh$ on $L$.

Recall that $T^*M$ is a Liouville manifold, that means there is a
(codimension zero) submanifold with boundary $K\subseteq T^*M$, such that
$\theta$ restricts to a contact form on $\partial K$ and there is a diffeomorphism
\be\label{K}
T^*M\setminus K\simeq \left[1,+\infty\right)\times \partial K
\ee
identifying $\theta$ with $r\cdot \theta|_{\partial K}$, where $r$ is the coordinate in $\left[1,+\infty\right)$.

To see this we can fix a metric on $M$ and take $K=D_1^*M$, where
$$D_\rho^*M =\{(q,p)\in T^*M\vert \langle p,p\rangle \leq \rho^2\}.$$

We will only consider Lagrangian submanifolds which are compact or asymptotically conic:
\begin{defn}
	An (embedded) Lagrangian submanifold $L$ in $T^*M$ is called asymptotically conic if, for $K$ as in (\ref{K}), if we denote $L^K:=L\cap K$, then
	$\partial L^K$ is Legendrian in $\partial K$ and we have a diffeomorpphism $L \setminus L^K \simeq \left[1,+\infty\right)\times L^K$.
\end{defn}

In fact, the only non compact Lagrangians we shall consider in the present paper are cotangent
fibers $T^*_qM$, $q\in M$, which are obviously conic.

We need to choose some auxiliary data to define the wrapped Floer cohomology.
We fix a Hamiltonian function $H:T^*M\to\mathbb{R}$ \emph{quadratic at infinity}, which means it is equal to $r^2$ on $T^*M\setminus K$. Note we only impose a restriction on $H$ outside of $K$ and so these Hamiltonians are different from the Tonelli Hamiltonians we consider in Section 7.

Given a pair of Lagrangian submanifolds $(L_0,h_0)$, $(L_1, h_1)$ we define a function on $\mathcal{P}(L_0,L_1)$ the space of paths $\gamma:[0,1]\to T^*M$ satisfying $\gamma(0)\in L_0$, $\gamma(1)\in L_1$:
\be\label{eq:AAH}
\mathcal{A}_H(\gamma)=\int_{[0,1]}( \gamma^*\theta - H\circ\gamma \, dt)+h_0(\gamma(0))-h_1(\gamma(1)).
\ee
We call this function the \emph{action functional}.

\begin{rem}\label{rem:sign} Here we use the sign convention of \cite{oh:jdg,kasturi-oh1} for the
	action functional, so that when $h_1 = h_0 = 0$, the action functional becomes
	the classical action functional in the mechanics literature. We remark that
	this definition is the negative of the one used in \cite{abouzaid2} but coincides with
	that of \cite{bernard1}.
	We regard the definition of the natural homological complex in Floer theory as the
	homological complex and then by reversing the flow, we identify the
	corresponding complex with the cohomological one. In this way, our definition of
	wrapped Floer homology can be identified with that of the cohomology in \cite{abouzaid2}.
\end{rem}

The first variation formula of this action functional is given by
\be\label{eq:1stvariation}
d\CA_H(\gamma)(\xi) = \int_{[0,1]} \omega(\dot\gamma - X_H(t,\gamma(t)), \xi(t))\,dt,
\ee
where $X_H$ is the Hamiltonian vector field associated to $H$, that is $\omega(X_H,\cdot)= dH$.

The following proposition immediately follows from \eqref{eq:1stvariation}.

\begin{prop}\label{critA}
	Let $\gamma\in\mathcal{P}(L_0,L_1)$, then $\gamma$ is a critical point of $\mathcal{A}_H$ if and only if it satisfies $\dot{\gamma}(t)=X_H(\gamma(t))$. We denote by $\mathcal{X}(L_0,L_1)$ the set of critical points
	pf $\CA_H$.
\end{prop}

Given the conditions imposed on the Lagrangians we have the following lemma (see \cite{abouzaid}).

\begin{lem}
	Given $a\in \mathbb{R}$ denote by $\mathcal{X}_{\leq a}(L_0,L_1)$ the set of critical points $\gamma$ of $\mathcal{A}_H$  with $\mathcal{A}_H(\gamma)\leq a$. This set is compact and after generic perturbation of $L_0$ (or $L_1$) it is finite.
\end{lem}

In order to define wrapped Floer homology as a $\mathbb{Z}$-graded abelian group we need to impose some topological restrictions on the Lagrangians.

From now on we assume that $M$ is orientable. Then we choose a volume form in $M$ and we complexify it (using a compatible almost complex structure). This gives a quadratic complex volume form $\eta$ on $T^*M$ (see details in \cite{abouzaid2} and Section 11 of \cite{seidel:book}). For each point $p$ in a Lagrangian $L$, we can evaluate $\eta/|\eta|$ on a basis for $T_p L$. This is independent of the choice of basis and therefore defines a map
$$\frac{\eta}{|\eta|}:L\longrightarrow S^1.$$
If the Maslov class of $L$ vanishes (see Section 11 of \cite{seidel:book} for details) we can choose a lift of this map to a real-valued function. We call such choice a grading on $L$.
By choosing gradings on $L_0$ and $L_1$ we can assign a Maslov index to each element $\gamma\in\mathcal{X}(L_0,L_1)$ which we denote by $|\gamma|\in \mathbb{Z}$.

We define the \emph{wrapped Floer complex} as

\be\label{eq:complex}
CW_*(L_0,L_1)= \bigoplus_{i}CW_i(L_0,L_1)=\bigoplus_{i}\bigoplus_{\substack{x\in\mathcal{X}(L_0,L_1)\\|x|=i}}x\cdot\mathbb{Z}.
\ee

That is $CW_i(L_0,L_1)$ is the free abelian group generated by Hamiltonian chords of Maslov index $i$.

Next we define the \emph{Floer differential}. For this we need to consider a family $\{J_t\}_{t\in[0,1]}$ of almost complex structures on $T^*M$ compatible with $\omega$. Additionally we assume that in the complement of $K$ each complex structure $J$ satisfies
$$\theta\circ J=dr.$$
We consider the Cauchy--Riemann equation in the space of maps $u:\mathbb{R}\times [0,1]\to T^*M$
\be\label{CR}
\begin{cases}
	\frac{\partial u}{\partial \tau}+J_t\left(\frac{\partial u}{\partial t} - X_{H}(u)\right)=0,\\
	u(\tau,0)\in L_0, \ u(\tau,1)\in L_1.
\end{cases}
\ee
Given $x_0,x_1\in\mathcal{X}(L_0,L_1)$ we denote by $\mathcal{M}(x_0,x_1)$ the set of maps $u$ satisfying the above equation and converging (exponentially) to $x_0$ at $-\infty$ and to $x_1$ at $+\infty$, quotiented by translations on the $\tau$-direction.
For generic $\{J_t\}$, the moduli spaces $\mathcal{M}(x_0,x_1)$ are smooth manifolds of dimension $|x_0|-|x_1|-1$.

To define orientations on these moduli spaces coherently we need to impose one more topological restriction on the Lagrangians. Denote by $b=\pi^*{w}_2(TM)\in H^2(T^*M,\mathbb{Z}_2)$ the pullback of the second Stiefel-Whitney class of $M$. We say a Lagrangian $L$ is \emph{relatively spin} if $b|_{L}=w_2(TL)$ and a choice of relative spin structure is defined to be a spin structure on the vector bundle $TL\oplus\pi^*(TM)\vert_{ L}$.

Relative spin structures on $L_0$ and $L_1$ then determine orientations on the moduli spaces $\mathcal{M}(x_0,x_1)$ (see \cite{abouzaid} for details).

When $|x_0|=|x_1|+1$, $\mathcal{M}(x_0,x_1)$ is a zero-dimensional oriented manifold and Gromov compactness implies that it is compact. Hence we define $\#\mathcal{M}(x_0,x_1)$ as the signed count of elements in $\mathcal{M}(x_0,x_1)$. Finally we define

$$\del :CW_i(L_0,L_1)\to CW_{i-1}(L_0,L_1)$$
as
$$\del (x_0)=\sum_{\substack{x_1\in\mathcal{X}(L_0,L_1)\\ |x_0|=|x_1|+1}}\#\mathcal{M}(x_0,x_1)\cdot x_1.$$

A standard argument implies the following
\begin{lem}
	The map $\del$ is a differential, i.e. $\del^2=0.$
\end{lem}

This differential has one additional property which follows from the following standard lemma.
(See  \cite{oh:jdg} or \cite{abouzaid2} for the proof.)

\begin{lem}
	If $\mathcal{M}(x_0,x_1)$ is non-empty, then $$\mathcal{A}_H(x_0)\geq\mathcal{A}_H(x_1).$$
\end{lem}

We define a filtration $\nu: CW_*(L_0,L_1)\to \mathbb{R}$ as follows: given $\alpha=\sum_xa_xx\ \in CW_*(L_0,L_1)$ we take $\nu(\alpha)=\max_{a_x\neq 0}\{\mathcal{A}_H(x)\}.$ The above lemma implies that
$$
CW_*^{\leq a}(L_0,L_1)=\{\alpha\in CW(L_0,L_1)|\nu(\alpha)\leq a\}
$$
is a subcomplex of $CW_*(L_0,L_1)$.

We define the \emph{wrapped Floer homologies}
$$
HW_*(L_0,L_1)=\ker \del/ \Im \del
$$
and
$$
HW_*^{\leq a}(L_0,L_1)=\ker \del^{\leq a}/ \Im \del^{\leq a}
$$
where $\del^{\leq a}$ is the restriction of $\del$ to the subcomplex $CW_*^{\leq a}(L_0,L_1)$.
Note that the inclusion $j_a:CW_*^{\leq a}(L_0,L_1)\hookrightarrow CW_*(L_0,L_1)$ induces a map
$$
(j_a)_*:HW_*^{\leq a}(L_0,L_1)\to HW_*(L_0,L_1).
$$

One can be also define the wrapped Floer cohomology $HW^*(L_1,L_0)$ and its filtration
$HW^*_{\geq a}(L_1,L_0)$ following the canonical construction of cohomology.
This can be canonically identified with the complex obtained by `reversing the flow'.
In the present case, the latter construction then will lead to the
wrapped Floer cohomology defined in \cite{abouzaid2}.

\subsection{Floer homology of an exact Lagrangian and a fiber}

The wrapped Floer complexes defined in the previous subsection for each pair of Lagrangians can be combined to define an A$_{\infty}$-category, $\mathcal{W}(T^*M)$ known as the \emph{wrapped Fukaya category} of $T^*M$.
%We describe it very briefly. Objects are the graded, relatively spin Lagrangians considered in the previous section, and morphisms between $L_0$ and $L_1$ are the  wrapped Floer complexes $CW^*(L_0,L_1)$. One can then define operations
%$$
%\mu^k:CW^*(L_k, L_{k-1})\otimes \ldots \otimes CW^*(L_1,L_0)\to CW^*(L_k,L_0)
%$$
%by counting solutions to a Cauchy Riemann equation similar to (\ref{CR}) but this time the domain is a disk with $(k+1)$ boundary punctures and the piece of the boundary between the $i$-th and the $(i+1)$-th puncture is mapped to $L_i$.
The full details of this construction are intricate and can be found in \cite{AS} and \cite{abouzaid}.

\begin{rem} We would like to note that in \cite{abouzaid} the complex
	$CW^*(L,L')$ is associated to the space of paths running from
	$L$ to $L'$ which is opposite to our convention. Our enumeration of
	the $L_i$'s in the pair $CW^*(L_i,L_{i-1})$ is the opposite to the one given in
	\cite{abouzaid}. Our present convention is consistent with that of \cite{fooo:book}.
	(More specifically see Section 2.3 \cite{fooo:book} and also Remark \ref{rem:sign} of the
	present paper.)
\end{rem}

The work of Abouzaid \cite{abouzaid}, \cite{abouzaid2} gives a complete description of the category $\mathcal{W}(T^*M)$. In \cite{abouzaid} the author proves that any cotangent fiber $T^*_qM$ generates this category. We will not make use of this statement in its entirety, we will just use a particular consequence of this fact proved in Appendix C of \cite{abouzaid2}.

\begin{thm}[Abouzaid]
	Assume $M$ is orientable and let $L$ be an exact, compact Lagrangian with vanishing Maslov class, then
	\begin{itemize}
		\item[(a)] L is relatively spin;
		\item[(b)] the Floer cohomology $HW^*(T^*_qM,L)\simeq \mathbb{Z}$ is free of rank 1.
	\end{itemize}
\end{thm}

Note that a cotangent fiber $T_q^*M$ has vanishing Maslov class and it is relatively spin since it is contractible and  $b_{|T_q^*M}=0$. Therefore part (a) implies that the Floer cohomology in (b) is well-defined. In fact the proof of the above theorem requires an extension of the wrapped Fukaya category that includes local systems on the Lagrangians - this is carried out in \cite{abouzaid2}.

There is another deep result which shows that the Maslov class condition in the above theorem is superfluous.
\begin{thm}[Abouzaid--Kragh]
	Let L be a compact exact Lagrangian in $T^*M$, then $L$ has vanishing Maslov class.
\end{thm}
This theorem is proved by Abouzaid in an Appendix to Kragh's paper \cite{kra}. The proof involves ideas from algebraic topology, namely the construction of Viterbo's transfer map on symplectic cohomology as a map of spectra.

Combining the above theorems we obtain the following corollary which will enable our construction of graph selectors in the next section.

\begin{cor}\label{wFc}
	Let $L$ be an exact, compact Lagrangian in $T^*M$ for $M$ orientable. Then the wrapped Floer cohomology $HW^*(T^*_qM,L)$ is well defined and
	$$HW^*(T^*_qM,L)\simeq\mathbb{Z}.
	$$
\end{cor}

Since $HW^*(T^*_qM,L)\simeq\mathbb{Z}$, we also conclude that
$HW_*(L,T^*_qM)\simeq\mathbb{Z}$.

\section{Graph selectors}

We will now construct a graph selector for any compact, exact Lagrangian $L$ in $T^*M$
when $M$ is orientable. The construction follows the ideas from \cite{oh:jdg} combined with Corollary \ref{wFc}.

\subsection{Choice of Hamiltonian $H$ adapted to $L$}

Recall the general action integral formula \eqref{eq:AAH}
$$
\CA_H(\gamma) = \int_{[0,1]} (\gamma^*\theta - H\circ \gamma\, dt) + h_0(\gamma(0)) - h_1(\gamma(1))
$$
associated to the pair
\beastar
(L_0,h_0) & = & (L,h);\quad L = \iota(N), h = S\circ \iota^{-1}\\
(L_1,h_1) & = & (T_q^*M, 0).
\eeastar
for a fiber $T_q^*M$ at $q \in M$. Here we note that
the constant function $h \equiv 0$ is a primitive of $T_q^*M$ since $\theta|_{T_q^*M} \equiv 0$.

For the discussion of our main interest in this paper, all the Lagrangian submanifolds
$L$ are contained in a compact subset $K$ say, in a disc bundle $D_R^*M$
for some sufficiently large constant $R > 0$. To make our graph selector of $L$ independent of the choice of
Hamiltonian $H$ given in $\CA_H$, we put the following condition on the support of $H$.
\begin{cond}\label{cond:suppH} Let $R > 0$ be as above. We assume that
	\be\label{eq:suppH}
	\supp H \subset T^*M \setminus D_R^*M.
	\ee
\end{cond}
Let $q\in M$ be such that $T^*_qM$ and $L$ are transversal. By our choice of $H$, $X_H =0$ on $D_R^*M$
and hence Proposition \ref{critA} implies that any $x\in \mathcal{X}(L,T^*_qM)$ is a constant path
associated to an intersection point in $L \cap T_q^*M$. Therefore we have a one-to-one correspondence
$$
\mathcal{X}(L,T^*_qM) \cong T^*_qM\cap L.
$$
Moreover, since $\theta$ vanishes on $T^*_qM$, we derive
\be\label{eq:AHh}
\mathcal{A}_H(x)=h(x)
\ee
from \eqref{eq:AAH}.

\subsection{Structure of the wave front of exact Lagrangians}

In this subsection, we give precise description of generic properties of
the wave front of exact Lagrangian $(L,h)$ in relation to the singularities
of the primitive $h$. We denote the wave front set associated to $(L,h)$ by
$$
WF(L,h) = \{(h(x), x) \in \R \times T^*M \mid x \in L\}
$$
which forms a Legendrian submanifold of the 1-jet bundle $J^1(M)\cong \R \times T^*M$.

Let $\pi_L: L \to M$ be the restriction to $L$ of the natural projection $\pi:T^*M \to M$. The \emph{caustic} of $L$ is the set of critical values of $\pi_L$, which we denote by $\text{\rm Caus}(L) \subset M$. Let
\be\label{eq:UL}
\mathcal{U}_L = M \setminus \text{\rm Caus}(L) \subseteq M.
\ee
Note that $\mathcal U_L$ is the set of $q\in M$ such that $T^*_qM$ is transversal to $L$.

Next consider a subset of $\mathcal{U}_L$ defined as
\be\label{eq:Cerf-regular}
\mathcal{U}_L^{Cf} = \{ \ q \in \mathcal{U}_L \ \ \vert \ \ h|_{T^*_qM \cap L} \ \textrm{is injective} \ \}
%E_2 & = & \textrm{pr}_2\left(\{(a,x) \in \R \times L \mid a = h(x), \, dh(x) = 0\}\right)
\ee
We remark that although the primitive $h$ is used in their definitions,
the sets do not depend on the choice of the primitives (for connected $L$). Points in this set are called \emph{Cerf-regular}, see \cite{oh:minimax}.

The following proposition is an application of Sard's theorem. A similar statement is proved in \cite{oh:minimax} in a much more complex context, but for convenience of the reader we give a proof here.

\begin{prop}\label{prop:Cerf}
	Let $L$ be an exact Lagrangian submanifold and $h$ a Liouville primitive. Then $\mathcal U_L^{Cf}$
	is an open, dense subset of $M$ of full measure. Furthermore, for any
	$q \in \mathcal U_L^{Cf}$, $\pi_L^{-1}(q)$ is a finite set $\{x_0,\ldots , x_{k_q}\}$ and there exists an open (connected) neighborhood $U_q$ of $q$ for which we have the decomposition
	$$
	\pi_L^{-1}(U_q) = \coprod_{i=0}^{k_q} V_{x_i}
	$$
	where each $V_{x_i}$ is a an open neighborhood of $x_i$, and we have
	\begin{enumerate}
		\item $\pi_L|_{V_{x_i}}: V_{x_i} \to U_q$ is a diffeomorphism for each $i = 0, \ldots, k_q$.
		\item The primitive $h$ restricts to an injective function on each fiber $\pi_L^{-1}(q')$
		for all $q' \in U_{q}$.
	\end{enumerate}
\end{prop}
\begin{proof}
	By definition $\mathcal U_L$ is the set of regular values of $\pi_L$, and hence it is open and has total measure by Sard's theorem. Compactness of $L$ then implies that $\pi_L^{-1}(q)$ is a finite set for any $q\in \mathcal U_L$. Again by definition, $\pi_L$ is a local diffeomorphism when restricted to $\mathcal U_L$, which immediately implies the existence of the neighborhoods $U_q$, $V_{x_0}, \ldots, V_{x_{k_q}}$ and property (i) for any point in $\mathcal U_L$. Property (ii) follows from the definition of $\mathcal U_L^{Cf}$.
	
	We are left with showing that $\mathcal U_L^{Cf} \subset \mathcal U_L$ has total measure.
	For this consider 1-forms $\varphi_i$ on $U_q$ for each $i=0,\ldots, k_q$, determined by the equations
	$(q', \varphi_i(q'))= \pi|_{V_{x_i}}^{-1}(q')$ for $i = 0, \ldots, k_q$. Denote $x_i(q') = (q', \varphi_i(q'))$.
	Then for each pair $0\leq i<j\leq k_q$ define the function $\delta_{ij}: U_q \to \R$ by
	$$
	\delta_{ij}(q')=h(q', \varphi_i(q')) - h(q', \varphi_j(q'))
	$$
	and denote its zero set by $\Delta_{ij}=\delta_{ij}^{-1}(0)$.
	Since $h$ is a Liouville primitive we compute
	\beastar
	d_{q'}(h(\cdot, \varphi_i(\cdot))) & = & dh|_{x_i(q')} \circ d\pi_{V_{x_i(q')}}^{-1} = \theta_{x_i(q')}\circ d\pi_{V_{x_i(q')}}^{-1}\\
	& = & p\left(d\pi \circ d\pi_{V_{x_i(q')}}^{-1}\right) = p(x_i(q')) = \varphi_i(q')
	\eeastar
	for each $i=0, \ldots, k_{q_0}$. Therefore for $i \neq j$,
	$$
	d\delta_{ij}(q') = \varphi_i(q')- \varphi_j(q') \neq 0
	$$
	where the last non-vanishing holds by the definition of $\varphi_i$'s. Hence we conclude that the $\Delta_{ij}$ are smooth hyper-surfaces. Then observe
	$$\mathcal U_L^{Cf} \cap U_q = U_q \setminus \bigcup_{i<j} \Delta_{ij},$$
	hence $\mathcal U_L^{Cf} \cap U_q$ is an open set of total measure which implies the desired result.
\end{proof}

%The following definition deserves a name.

%\begin{defn}\label{defn:Cerf-regular} Let $(N,\iota,S)$ be an exact Lagrangian brane and
%$L = \iota(N)$ be as above. We call any point $q \in \mathcal U_L^{Cf} $ \emph{Cerf-regular}
%if it satisfies the property given in Proposition \ref{prop:Cerf}.
%\end{defn}

\subsection{Floer-theoretic graph selectors}

Consider an exact Lagrangian submanifold $(L,h)$ as before.
For each $q \in M$,
we define the subset of $\mathbb{R}$,
\be\label{eq:spectrum}
\Spec(L,h;q) = \{ h(x) \in \R \mid x \in L \cap T^*_qM\}.
\ee
We call it the \emph{spectrum} of $(L,h;q)$.
By compactness of $L$, $\Spec(L,h;q)$ is a finite set whenever
$L$ intersects transversely $T_q^*M$.
%For a general point $q \in M$,
%Sard's theorem implies that $\Spec(L,h;q)$ is still a measure zero subset.
%(See section 12 \cite{oh:book}, \cite{oh:alan} for more detailed discussion
%on general construction of spectral invariants.)

Now we introduce our main

\begin{defn}
	For $q \in \mathcal{U}_L$ we define the spectral invariant of the pair $T^*_qM,\ L$:
	$$
	\rho(L,q)=\min\left\{\lambda\in\mathbb{R}\ |\ (j_\lambda)_*:HW_*^{\leq\lambda}(L,T^*_{q}M)\to HW_*(L,T^*_qM) \textrm{ is surjective}\right\}.
	$$
\end{defn}
It can be seen, using an argument similar to the one in the construction of (\ref{cobordism_arg}) in Section 7, that these values do not depend on the choice of Hamiltonian $H$ as long as it satisfies the support condition given in Condition \ref{cond:suppH}. Since we do not use this fact we will not give a complete proof.
Also $\rho(L,q)$ loosely depends on the choice of $h$, not just on $L$. Since
the $h-h' \equiv const$ for any two generating functions $h, \, h'$ of (connected) $L$,
we ignore this dependence.
\begin{rem}
	We would like to remark that this definition makes sense for any point $q$ not necessarily in $\mathcal{U}_L$, but for these we would have to consider a Hamiltonian that does not satisfy Condition \ref{cond:suppH}.
\end{rem}

Note that since $HW_*(L,T^*_qM)\simeq\mathbb{Z}$, the above condition is equivalent to defining $\rho (L,q)$ as
$$\min_{[\alpha]=[\varphi_*(1)]}\{\nu(\alpha)\}$$ where $\varphi:\mathbb{Z}\to CW_*(L,T^*_qM)$ is any homomorphism that induces an isomorphism on homology.
It follows from this interpretation of $\rho(L,q)$ and
the support hypothesis \eqref{eq:suppH} put on $H$, that the mini-max value $\rho(L,q)$ is realized by the action of some
$x_q \in L \cap T^*_qM$, that is, $h(x_q)=\rho(L,q)$. So we conclude
$$
\rho(L,q) \in \Spec(L,h;q)
$$
for all $q \in \mathcal{U}_L$. Therefore
$(\rho(L,q), x_q) \in WF(L,h)$ for all $q \in M$. (Such property is called the \emph{tightness} of the
mini-max value, for example see Definition 21.5.2 \cite{oh:book}.)

We remark that a priori the definition of $\rho(L,q)$ depends on the family of almost complex structures $\{J_t\}$, used to define the Floer differential. But, by the same argument performed in \cite[Lemma 6.3]{oh:jdg}, we can prove the following

\begin{lem}\label{lem:rhoLq}
	The value $\rho(L,q)$ does not depend on the choice of $\{J_t\}$.
\end{lem}
This enables us to define the \emph{basic phase function}  (or Floer theoretic graph selector) $f_L:\mathcal{U}_L\subseteq M \to \mathbb{R}$ by $f_L(q)=\rho(L,q)$.

The following is the main theorem of this subsection and is the counterpart of Theorem 9.1 in \cite{oh:jdg}. Note that, in particular, this proves Theorem \ref{gs} in the Introduction.

\begin{thm}\label{floer gs}
	Let $L$ be a compact exact Lagrangian and define $f_L$ as above. We can extend $f_L$
	to a Lipschitz function on $M$. If $q \in \mathcal{U}^{Cf}_L$ then $f_L$ is smooth at $q$ and we have
	\be\label{q,dfq}
	(q,d f_L(q))\in L\subseteq T^*M\textrm{ and } f_L(q)= h(q,df_L(q)).
	\ee
	Therefore, since $\mathcal{U}^{Cf}_L$ is an open set of total measure by Proposition \ref{prop:Cerf}, $f_L$ is a graph selector for $L$.
\end{thm}

The proof of this theorem will occupy the rest of this subsection.
We start with the following proposition whose proof we postpone to Section \ref{sec:lipschitz}.

\begin{prop}\label{prop1 GS}
	Equip $M$ with a metric and denote by $d$ the associated distance function. Consider $q_0,q_1\in \mathcal U_L\subseteq M$ so that $\rho (L,q_0)$ and $\rho(L,q_1)$ is defined. Then
	$$\vert f_L(q_0)-f_L(q_1)\vert\leq C\cdot d(q_0,q_1),$$ for a constant $C>0$ depending only on $L$ and the Hamiltonian $H$.
\end{prop}

We can now extend $f_L$ uniquely to a Lipschitz function on $M$ and prove the following

\begin{lem}\label{lem1 GS}
	For any $q\in M$, there exists $x_q\in L\cap T^*_qM$ such that $f_L(q)=\rho(L,q)=h(x_q)$.
\end{lem}
\begin{proof}
	This is obvious from the construction for points $q \in \mathcal{U}_L$.
	For the other points it follows from an easy continuity argument using denseness of $\mathcal{U}_L$ and compactness of $L$.
	We omit the easy details.\end{proof}

The following stability of
tightness of the values $\rho(L,q)$ is a crucial ingredient in the
chain level Floer theory entering in the study of spectral invariants.
(See \cite{oh:minimax}, especially section 3 thereof, for the illustration of such a usage.)

\begin{prop}\label{prop:tightness}
	Consider $q_0 \in \mathcal U_L^{Cf}$ and take $x_0 \in L \cap T_{q_0}^*M$ such that $h(x_0)=f_L(q_0)$. Let $V_{x_{0}}$ and $U_{q_0}$ be the neighborhoods provided by Proposition \ref{prop:Cerf}
	and for each $q \in U_{q_0}$ denote $x_0(q): = \pi|_{V_{x_0}}^{-1}(q)$. Then we have
	$$
	f_L(q) = h(x_0(q)).
	$$
	Hence, in particular, $f_L$ is a smooth function on $U_{q_0}$.
\end{prop}
\begin{proof}
	Let $\{x_0, \ldots, x_k\}$ be the pre-image of $q_0$, consider the data provided by Proposition \ref{prop:Cerf} and denote $x_{i}(q)= \pi|_{V_{x_i}}^{-1}(q)$, for $i=0,\ldots, k$, so that $x_i(q_0)=x_i$.
	We consider the subset
	$$
	U^=_{q_0} = \{ q \in U_{q_0} \mid f_L(q) = h(x_0(q)) \}.
	$$
	This set is nonempty since $q_0 \in U^=_{q_0}$. By continuity of $f$ and $h\circ x_0$, it is closed
	in $U_{q_0}$. Now we show that it is also open in $U_{q_0}$.
	Let $q' \in U^=_{q_0}$.
	By assumption $h$ is injective in $\pi^{-1}(q')=\{x_0(q'),\ldots,x_k(q')\}$,
	hence there exists $\epsilon>0$ such that
	$$
	|h(x_i(q')) - h(x_0(q'))|> \epsilon
	$$
	for all $i\neq 0$. Since $q' \in U^=_{q_0}$, we have $f_L(q') = h(x_0(q'))$.
	Then by continuity of $f_L$ and $h(x_i(-))$ (for each $i$), we can take a neighborhood $U'_{q'} \subset U_{q_0}$ of $q'$,
	so that $|f_L(q)- f_L(q')|< \epsilon/2$ for all $q \in U_{q'}'$
	and that $|h(x_i(q))-f_L(q')|>\epsilon/2$ for all $q \in U'_{q'}$ and $i\neq 0$.
	By Lemma \ref{lem1 GS}, for each $q$ there is some $i$ such that $f_L(q)=h(x_i(q))$. Hence the previous inequalities force $f_L(q) = h(x_0(q))$ for all $q\in U'_{q'}$ and so $U'_{q'} \subset U_{q_0}^=$. This proves that the set $U^=_{q_0}$ is open.
	Therefore by connectedness of $U_{q_0}$, $U^=_{q_0} = U_{q_0}$, that is,
	$f_L(q) = h(x_0(q))$ for all $q \in U_{q_0}$.
\end{proof}

\vspace{.1cm}
\begin{proof}[Proof of Theorem \ref{floer gs}]
	Proposition \ref{prop1 GS} enables us to extend uniquely $f_L$ to a Lipschitz function on $M$.
	Then Proposition \ref{prop:tightness} implies that $f_L$ is a smooth function on $\mathcal U_L^{Cf}$ and that for $q\in \mathcal U_L^{Cf}$
	and $x \in L$ such that $f_L(q)= h(x)$ we have
	$$
	(q,df_L(q))=x \in L\cap T^*_q M.
	$$
	To see this, we denote $x =(q,p)$ and compute
	\begin{align*}
	df_L(q)&=dh \circ d \pi^{-1}_{V_{x}}
	=(\theta|_L)_{x}\circ d\pi^{-1}_{V_{x}}
	=p(d\pi \circ d\pi^{-1}_{V_{x}})=p,
	\end{align*}
	where we have used that $dh=\theta\vert_{ L}$ and the definition of $\theta$. This finishes the proof of Theorem \ref{floer gs}.
\end{proof}

\begin{rem}
	There might be points $q\in M \setminus \mathcal{U}_L^{Cf}$ where $f_L$ is still differentiable. But for these Equation (\ref{q,dfq}) might not hold. However, it follows from Lemma \ref{nonsmooth} that $(q,d f_L(q))$ lies in the fiberwise convexification of $L$.
	This fact is the key motivation for the notion of generalized graph selector.
\end{rem}

\section{Generalized graph selectors}

In this subsection we weaken the notion of graph selector, following \cite{bernard-santos},
and prove that Lipschitz-exact Lagrangians admit \emph{generalized graph selectors}.

We start by introducing some notation. For a subset $C\subset T^*M$ we denote by $\widehat{C}$ its fiberwise convexification, that is $\widehat{C}_q:=\widehat{C} \cap T^*_q M$ is the convex hull in $T^*_q M$ of $C\cap T^*_q M$. Recall that a point in a convex set $C$ is said to be \emph{extremal} if it is not in the interior of a line segment whose endpoints lie in $C$. Observe that an extremal point in $\widehat{C}$ must lie in $C$.

\begin{defn}\label{def:gen_slector}
	Let $(N,\iota, S)$ be an Lipschitz-exact Lagrangian brane and $L = \iota(N)$.
	Denote $h = S\circ \iota^{-1}$. A \emph{generalized graph selector} for $L$ is a
	Lipschitz function $f: M\to \mathbb{R}$ satisfying the following conditions,
	\begin{itemize}
		\item[\textbf{(a)}] if $f$ is differentiable at $q$ then $(q,df(q))\in \widehat{L}_q$;
		\item[\textbf{(b)}] for each point $q$ where $f$ is differentiable and $df(q)$ is an extremal point of
		$\widehat{L}_q$ we have $f(q)=h(q,df(q))$.
	\end{itemize}
	
\end{defn}

\begin{thm}\label{gen gs}
	Any Lipschitz-exact Lagrangian brane $(N,\iota, S)$, with $N$ compact, $L = \iota(N)$ admits a generalized graph selector.
\end{thm}
The proof of this theorem follows closely the proof of Proposition 3 in \cite{bernard-santos}
using the graph selector for $L$ constructed in Theorem \ref{floer gs}.
We present the main steps of the proof for the sake of completeness.

The first thing we need is a generalization of a result of non-smooth analysis.

\begin{lem}[Lemma 2 \cite{bernard-santos}]\label{nonsmooth}
	Let $f_k:M\to \mathbb{R}$ be a sequence of equi-Lipschitz functions converging uniformly to a Lipschitz function $f:M\to\mathbb{R}$. Let $V\subset M$ be a set of total Lebesgue measure such that each one of the $f_k$ is differentiable in $V$. We define $\Lambda_q\subset T^*_qM$ to be the set of all limits of subsequences of sequences of the form $(q_k,df_k(q_k))$ with $\{q_k\}\subset V$ and $q_k\to q$. If $f$ is differentiable at the point $q$ then $(q, df(q))\in \widehat\Lambda_q$.
\end{lem}
\vspace{.05cm}
\begin{proof}[Proof of Theorem \ref{gen gs}]
	Since $(N, \iota, S)$ is an Lipschitz-exact Lagrangian brane we have equi-Lipschitz sequences $S_k:N \to\R$ of smooth functions and $\iota_k: N \to T^*M$ of smooth embeddings such that
	$$
	\iota_k^*\theta_k=d S_k, \ \textrm{and } S_k\to S,\ \iota_k\to \iota\textrm{ uniformly}.
	$$
	As before denote $L_k=\iota_k(N)$ and consider the Liouville primitive $h_k = S_k\circ \iota_k^{-1}: L_k \to \R$.
	
	Theorem \ref{floer gs} gives, for each $(N,\iota_k,S_k)$, a graph selector $f_{L_k}$ of $L_k = \iota_k(N)$. Consider $V= \bigcap_{k} \mathcal{U}_{L_k}^{Cf}$, which is a full measure set where all the $f_{L_k}$ are differentiable,
	\begin{equation}\label{smoothgr}
	(q,df_{L_k}(q))\in L_k \  \ \textrm{and} \ \   f_{L_k}(q)=h_k(q,df_{L_k}(q)),
	\end{equation} for all $q\in V$ and all $k$.
	Since $(q,df_{L_k}(q))\in L_k$ and $\iota_k \to \iota$ uniformly, there exists a
	constant $C> 0$ independent of $k$'s but depending only on $L$ such that
	$$
	\max_{q \in M} |df_{L_k}(q)| \leq C
	$$
	and hence $f_{L_k}$ are equi-Lipschitz.
	Then Ascoli-Arzela Theorem
	implies that the sequence $f_{L_k}$ converges uniformly to a Lipschitz limit $f$.
	
	One can now see that $f$ is a generalized graph selector for $(N, \iota, S)$.  First we prove that $\Lambda_q\subset L_q$ which implies $\widehat\Lambda_q\subset \widehat L_q$. A point $(q,p)$ in  $\Lambda_q$ is a limit of  a subsequence of the form $\{(q_k,df_{L_k}(q_k))\}$ whith $q_k\in V$, $q_k\to q$. By (\ref{smoothgr}) we have $(q_k,df_{L_k}(q_k))\in L_k$. Hence there is $x_k$ such that $(q_k,df_{L_k}(q_k))=\iota_k(x_k)$ and we can assume $x_k$ has a limit $x$. Since $\iota_k$ converges uniformly to $\iota$, we have  $(q,p)=\lim \iota_k(x_k)=\iota(x)\in L_q$. Now applying Lemma \ref{nonsmooth} to the sequence $f_{L_k}$ we conclude that,  if $q$ is a point of differentiability of $f$, $(q, df(q))$ lies in $\widehat\Lambda_q$. Putting these two together shows part (a) of the generalized graph selector statement in its definition.
	
	It remains to see part (b) that $f(q)=h(q,df(q))$ when $df(q)$ is extremal in $\widehat L_q$. Because $\widehat\Lambda_q\subset \widehat L_q$, $df(q)$ is also extremal in $\widehat\Lambda_q$ and hence belongs to $\Lambda_q$. By the definition of $\Lambda_q$ this means there exists a sequence $\{q_k\}\subset V$ such that $(q_k,df_{L_k}(q_k))\to (q, df(q))$. The equality $f(q)=h(q,df(q))$ now follows from taking the limit $k\to \infty$ in (\ref{smoothgr}).
\end{proof}

\section{Applications to Hamiltonian dynamics}

In this section we prove our two applications Theorem \ref{aplic1} and Theorem \ref{aplic2}, following \cite{bernard-santos}. We will concentrate on the differences and when the proofs are straightforward generalizations we omit them.

Let $M$ be a closed orientable $n$-manifold and fix an autonomous $C^2$ Tonelli Hamiltonian $H:T^*M\to \mathbb{R}$ (recall this means positive definite Hessian and superlinear in each fiber). As in the introduction let $\mathcal{L}$ be the set of compact Lipschitz-exact Lagrangians and given $L=\iota(N) \in \mathcal{L}$ we define the maximal invariant subset of $L \cap \{H=a\}$ as
$$
\mathcal{I}_a^*(L):=\bigcap_{t\in\mathbb{R}}\phi_H^t\left(L\cap \{H=a\}\right).
$$
We recall the following standard definition from weak KAM Theory.
\begin{defn} Let $H$ be a Tonelli Hamiltonian on $T^*M$. A Lipschitz function $f:M \to \R$ is called a
	\emph{sub-solution} of the (autonomous) Hamilton-Jacobi equation
	\be\label{HJ-equation}
	H(q,df(q)) = a
	\ee
	if it satisfies $H(q,df(q)) \leq a$ almost everywhere on $M$.
\end{defn}
The following regularization result of Bernard,
which is proved in this context in Appendix B of \cite{bernard-santos}, plays a significant role
in the proof of Theorem \ref{aplic1} similarly as in \cite{bernard-santos}.

\begin{thm}[Theorem 3 \cite{bernard-santos}]\label{lem:bernard}
	Given $a \in \R$, if $f$ is a Lipschitz sub-solution of \eqref{HJ-equation}, then
	there exists a $C^{1,1}$ sub-solution $\widetilde f$ such that
	$\mathcal I_a^*(\Gamma_f) = \mathcal I_a^*(\Gamma_{\widetilde f})$.
\end{thm}
The statement of this theorem is slightly different from that of
Theorem 3 \cite{bernard-santos} but is equivalent thereto by the argument given in
p.170 \cite{bernard-santos}.
Now we have the following
\begin{thm}\label{main aubry mather}
	Let $H$ be a Tonelli Hamiltonian and $a\in \mathbb{R}$, then for each $(N,\iota,S)\in \mathcal{L}$ such that $L=\iota(N)$ is contained in the energy sub-level $\{H\leq a\}$, there exists $\Gamma\in \mathcal{G}$ contained in the same energy sub-level $\{H\leq a\}$ and such that $$\mathcal{I}^*_a(\Gamma)=\mathcal{I}^*_a(L).$$
\end{thm}

\begin{proof}
	The proof follows step-by-step the proof of Theorem 1 in \cite{bernard-santos} with the obvious difference that we need to use the new more general version of the generalized graph selector given in Proposition \ref{gen gs}. Because we start with a compact Lipschitz-exact Lagrangian $(N,\iota, S)$ with $L=\iota(N)\subset \{H\leq a\}$, Theorem \ref{gen gs} gives
	a generalized graph selector $f$. Since $H$ is fiberwise convex we have that, whenever defined, $\Gamma_f \subset \{H\leq a\}$.
	Therefore $f$ is a Lipschitz sub-solution of \eqref{HJ-equation}.
	Then Theorem \ref{lem:bernard} implies that there exists a $C^{1,1}$ sub-solution $\tilde f$ such that $\Gamma_{\tilde f}\subset \{H\leq a\}$. One then proves that the graph of $d\tilde f$ is what we are looking for, that is $\mathcal{I}^*_a(\Gamma_{\tilde f})=\mathcal{I}^*_a(L)$. The proof of this statement is now the same as the case covered in Section 3 (and page 170) of \cite{bernard-santos}. It uses standard techniques in weak KAM theory so we do not repeat it.
\end{proof}

We can now prove Theorem \ref{aplic1}.

\begin{proof}[Proof of Theorem \ref{aplic1}]
	We first prove that $\alpha_{\mathcal{E}}(H)=\alpha_{\mathcal{G}}(H)$. Recall that $$\alpha_\mathcal{E}(H):=\inf_{L\in\mathcal{E}}\max_{(q,p)\in L}H(q,p),$$
	$$\alpha_{\mathcal{G}}(H):=\inf_{L\in\mathcal{G}}\max_{(q,p)\in L}H(q,p).$$
	The inequality $\alpha_{\mathcal{E}}(H)\leq \alpha_{\mathcal{G}}(H)$ is true simply because $\mathcal{G}\subset \mathcal{E}$. Now suppose that $\alpha_{\mathcal{E}}(H) < \alpha_\CG(H)$.
	Then there is a $c\in ]\alpha_{\mathcal{E}}(H),\alpha_{\mathcal{G}}(H)[$ such that the set $\{H\leq c\}$ contains an element $L\in\mathcal{E}$. Then by Theorem \ref{main aubry mather} there exists a  $\Gamma\in\mathcal{G}$, $\Gamma\subset\{H\leq c\}$ which contradicts the definition of $\alpha_{\mathcal{G}}(H)$. Combining the two we have
	proved $\alpha_{\mathcal{E}}(H)= \alpha_{\mathcal{G}}(H)$.
	
	The coincidence of the definitions for the Aubry and Ma\~n\'e sets is now a direct consequence of the definitions themselves, the coincidence of the critical values and Theorem \ref{main aubry mather}.
	
	The remaining equalities, for when $L \in\mathcal{L}$, are proven in the exact same way since Theorem \ref{main aubry mather} applies to these Lagrangians.
\end{proof}

Before we prove the second application we need

\begin{lem}\label{inv}
	Consider $(N,\iota,S) \in \mathcal{L}$ such that $L=\iota(N)$ is invariant under the flow of a Tonelli Hamiltonian $H$. Then there exists $e$ such that $L \subset\{H=e\}$.
\end{lem}
\begin{proof}
	Since $\iota$ is Lipschitz and $H$ is $C^2$, $H\circ \iota$ is Lipschitz and hence differentiable almost everywhere. Proposition \ref{prop:iota*omega=0} implies there is a set $Z$ of total measure where $\iota^*\omega=0$ and both $\iota$ and $S$ are differentiable. Let $x\in Z$, then
	$$d(H\circ\iota)(x)=dH(\iota(x))d\iota(x)=\omega(X_H(\iota(x)),d\iota(x)).$$
	Since $\phi_H^t(L)\subset L$ we know $X_H(\iota(x))\subset T_{\iota(x)}L$.
	Therefore , since $\iota^*\omega=0$ in $Z$, we have
	$$
	\omega(X_H(\iota(x)),d\iota(x))=0.
	$$
	This implies that $H\circ \iota$ is a constant function, which means that there exists $e$ such that $L\subset\{H=e\}$.
\end{proof}	

\begin{proof}[Proof of Theorem \ref{aplic2}]
	Since $L=\iota(N)$ is invariant under the flow of $H$, Lemma \ref{inv} gives $e$ such that $L \subset\{H=e\}$.
	Therefore, since $\mathcal{I}_e^*(L)$ is the maximal invariant subset of $L\cap \{H=e\}$ and $L$ is invariant we have
	$$
	\mathcal{I}_e^*(L)=L.
	$$
	Now by Theorem \ref{main aubry mather} there exists a $\Gamma\in\mathcal{G}$ such that
	$$
	\mathcal{I}_e^*(L)=\mathcal{I}_e^*(\Gamma)\subset\Gamma.
	$$
	Hence $L \subset \Gamma$ and since $M$ is connected  we conclude $L =\Gamma$.
\end{proof}

\section{Lipschitz continuity of the graph selector}
\label{sec:lipschitz}

In this section, we give the proof of Proposition \ref{prop1 GS}.
The proof follows the ideas of Section 12.5 \cite{oh:book} restricted to the case of submanifold $S = \{q\}$, but now in the wrapped context. The main step is to construct a chain map $\Phi: CW_*(L, T_{q_0}^*M) \to CW_*(L,T_{q_1}^*M)$ which induces an isomorphism in homology.
\begin{rem}
	Here we provide full details of this proof because Section 12.5 \cite{oh:book} (and \cite{oh:jdg}) is written in the
	unwrapped setting and because the construction of the homotopy map used in Section 12.5 \cite{oh:book} (which follows Nadler's argument \cite{nadler}) is different from that of the original construction of \cite{oh:jdg}).
	Also some of the details of this proof are not given in \cite{oh:book}.
\end{rem}

We start by taking a minimizing geodesic $c:[0,1] \to M$ from $q_0$ to $q_1$ with length $\ell = d(q_0,q_1)$.
We then consider a small generic perturbation of the path, again denoted by $c$, and extend the derivative $c'$
into an ambient vector field $X$ supported in a small neighborhood $V$ of the image of $c$.
We may choose $X$ so that
\be\label{eq:max|X|}
\max_{q \in M}|X(q)| \leq \ell + \epsilon
\ee
where $\epsilon> 0$ can be chosen arbitrarily small. Consider the autonomous Hamiltonian $G^0=\langle p, X(q)\rangle$  and denote by $\phi_{G^0}^s$ the associated flow. Observe that this defines a Hamiltonian isotopy from  $T_{q_0}^*M$ to $T_{q_1}^*M$, in fact $\phi_{G^0}^s(T^*_{q_0}M)=T^*_{c(s)}M$.

We will deform this isotopy to another isotopy from $T^*_{q_0}M$ to $T^*_{q_1}M$
by expressing it as the composition of two isotopies in the following way.
We choose $R_1> R+1$ (as defined in \eqref{cond:suppH}) such that $H$ is quadratic outside $D^*_{R_1}M$.  Let $\tilde{\chi}:\mathbb{R}\to [0,1]$ be a cut-off function satisfying
$$
\widetilde{\chi}(t)=\begin{cases} 0, \quad & t\leq R\\
1, \quad & t\geq R_1
\end{cases}
$$
and set $\chi = 1 -\widetilde \chi$.
Let $\phi_{\tilde{G}}^s$ be the Hamiltonian flow associated to
$
\tilde{G}(q,p) = \widetilde \chi(|p|) \langle p, X(q)\rangle
$
and denote $\widetilde{N}_s=\phi_{\tilde{G}}^s(T^*_{q_0}M)$. From the construction it follows that $\widetilde{N}_s$ is constant on $D^*_{R}M$. Similarly we define the Hamiltonian
\be\label{eq:G}
G=\chi(|y|) \langle p, X(q)\rangle
\ee
and denote by $\phi_{G}^s$ the associated Hamiltonian flow. We define $N_s=\phi_{G^1}^s(\widetilde{N}_1)$ and note that $N_s$ is constant on $T^*M \setminus D^*_{R_1}M$ and coincides with
$T^*_{c(s)}M$ on $D^*_{R}M$, moreover $N_1=T^*_{q_1}M$ and $N_0 = \widetilde N_1$.

Note that, since $L\subset D^*_{R}M$, the intersection points $L \cap \tilde{N}_s$ are independent of $s\in [0,1]$. Therefore there is an obvious isomorphism of vector spaces
\be\label{cobordism_arg}
\Phi_0: CW_*(L, T_{q_0}^*M) \to CW_*(L, \tilde{N}_{1}).
\ee
Furthermore $\Phi_0$ is a map of chain complexes. To see this we consider the following parametrized moduli spaces
$$
\mathcal{M}^{par}(x_0,x_1) := \bigcup_{s\in [0,1]} \{s\} \times \mathcal{M}^s(x_0,x_1),
$$
where $\mathcal{M}^s(x_0,x_1)$ is the moduli space defined in (\ref{CR}) but with boundary conditions in $L$ and $\tilde{N}_s$. For a generic choice of the almost complex structure $\mathcal{M}^{par}(x_0,x_1)$ becomes a compact smooth one-dimensional cobordism between $\mathcal{M}^0(x_0,x_1)$ and $\mathcal{M}^{1}(x_0,x_1)$. Thus we conclude $\#\mathcal{M}^0(x_0,x_1)= \# \mathcal{M}^{1}(x_0,x_1)$ and therefore $\Phi_0$ is a map of chain complexes leaving the filtration
unchanged.

Now we will define a chain map
$\Phi_{1}: CW_*(L, N_0) \to CW_*(L,T_{q_1}^*M)$ by considering the
Cauchy-Riemann equation with \emph{moving boundary conditions}.

We fix a smooth elongation function $\rho:\R \to [0,1]$ such that
$$
\rho(\tau) = \begin{cases} 0 \quad & \tau \leq 0, \\
1 \quad & \tau \geq 1,
\end{cases}
$$
and $\rho'(\tau) > 0$ on $(0,1)$. In particular, we have $\supp \rho' \subset [0,1]$.

Now we consider
\be\label{eq:movingCR}
\begin{cases}
	\frac{\partial u}{\partial \tau}+J_t\left(\frac{\partial u}{\partial t} - X_{H}(u)\right)=0,\\
	u(\tau,0)\in L, \ u(\tau,1)\in N_{\rho(\tau)}.
\end{cases}
\ee
Given $x_0\in\mathcal{X}(L,N_0)$ and $x_1\in\mathcal{X}(L,T^*_{q_1}M)$ we denote by $\mathcal{N}(x_0,x_1)$ the set of maps $u$ satisfying the above equation and converging (exponentially) to $x_0$ at $-\infty$ and to $x_1$ at $+\infty$. (Here there is no $\R$-action and so we do not quotient out the moduli space).
It again follows from the (strong) maximum principle proved in Section 7.3 \cite{AS} (see alternatively \cite[Lemma 3.3]{abouzaid3}) that solutions of (\ref{eq:movingCR}) satisfy a $C^0$ bound. More precisely we have the following lemma.
\begin{lem}\label{prop:C0bound}
	If $u\in \mathcal{N}(x_0,x_1)$ then $u(\tau, t)\in D^*_{R_1}M$ for  any $(\tau,t)$.
\end{lem}
This is stated in \cite{AS} just for the Cauchy--Riemann equations with fixed boundary conditions, but because our boundary conditions are fixed outside $D^*_{R_1}M$ the same proof applies.

For generic $\{J_t\}$,  when $|x_0|=|x_1|$, the moduli space $\mathcal{N}(x_0,x_1)$ is a smooth oriented manifold of dimension $0$. Gromov compactness together with the above $C^0$ bound then implies that $\mathcal{N}(x_0,x_1)$ is compact and therefore we can define the signed count $\#\mathcal{N}(x_0,x_1)$.
We then define the map
$$
\Phi_{1}(x_0) = \sum_{\substack x_1 \in \CX(L,T_{q_1}^*M),
	\ \ |x_0|=|x_1|} \#\mathcal{N}(x_0,x_1) \cdot x_1.
$$
The proof that this is a chain map which induces an isomorphism in homology is a standard argument in Floer theory that applies to both the wrapped or unwrapped context. We refer the reader to Section 12.5 \cite{oh:book} for complete details.
Putting these maps together we obtain a chain map
$\Phi = \Phi_1\circ \Phi_0: CW_* (L, T^*_{q_0} M) \to CW_* (L,T^*_{q_1} M)$,
which induces an isomorphism $\Phi_*$ in homology.

Now we study the action of $\Phi$ on the filtration $\nu$.

\begin{lem}\label{prop:action-diff}
	Given $x_0\in \CX(L,N_{0})$, $x_1\in \CX(L,T_{q_1}^*M)$ and $u \in \mathcal{N}(x_0,x_1)$ we have
	$$
	\CA_{H}(x_1) - \CA_{H} (x_0) = - \int_\R \int_{[0,1]} \Big|\dudt - X_H(u)\Big|^2_{J_t}\,dt\, d\tau
	+ \int_\R \rho'(\tau) G(u(\tau,1)) \,d\tau,
	$$
	where $|-|_{J_t}$ is the induced metric $\omega( - ,J_t-)$.
\end{lem}

Assuming this lemma for the moment, we complete the proof of Proposition \ref{prop1 GS}.
Lemma \ref{prop:action-diff} immediately implies
$$
\CA_{H}(x_1) - \CA_{H} (x_0) \leq \int_{-\infty}^\infty \rho'(\tau) G(u(\tau,1)) \,d\tau \leq \max_{\tau \in \R} G(u(\tau,1)).
$$
It follows from the definition $G(q,p) = \chi(|p|) \langle p, X(q)\rangle$ and \eqref{eq:max|X|} that
$$\max_{\tau \in \R} G(u(\tau,1))\leq R_1 (d(q_0,q_1) + \epsilon).$$
This gives the following key inequality
\be\label{eq:key<}
\CA_{H}(x_1) - \CA_{H} (x_0) \leq  R_1 (d(q_0,q_1) + \epsilon).
\ee

The rest of the argument is standard in the study of any type of spectral invariant. The chain map $\Phi_0$ preserves the filtration and therefore we only need to study what happens for $\Phi_1$. The inequality \eqref{eq:key<} implies that if $\nu(x_0)\leq \lambda$ then $\nu(\Phi_1(x_0))\leq \lambda + \delta$, where $\delta=R_1 (d(q_0,q_1)+\epsilon)$. Therefore the map $\Phi_1$ induces the following commutative diagram
$$
\xymatrix{
	HW_* ^{\leq \lambda } (L, N_0)\ar[r]^{(j_\lambda)_*} \ar[d]^{(\Phi_1)_*} &  HW_* (L, N_0)
	\ar[d]^{(\Phi_1)_*}  \\
	HW_* ^{\leq\lambda +\delta} (L,T^*_{q_1} M)\ar[r]^{\hspace{.5cm}(j_{\lambda+\delta})_*} & HW_* (L,T^*_{q_1} M)}.
$$
Since $(\Phi_1)_*$ on the right hand side is an isomorphism, we conclude
$$
\rho (L,q_1) \leq \rho (L,q_0) + R_1 (d(q_0,q_1) +\epsilon).
$$
Since $\epsilon$ is arbitrary we conclude $\rho (L,q_1) \leq \rho (L,q_0) + R_1 d(q_0,q_1)$. By changing the role of $q_0$ and $q_1$, we prove the other side of the
inequality, which leads to
\be\label{eq:rhoSdifference}
|f_L(q_1) - f_L(q_0)| \leq  R_1 d(q_0,q_1),
\ee
since by definition $f_L(q) = \rho(L,q)$. This completes the proof of the proposition modulo the proof of Lemma \ref{prop:action-diff}.

\begin{proof}[Proof of Lemma \ref{prop:action-diff}]
	We define the following parametrized version of the action functional (defined in \eqref{eq:AAH}):
	\be
	\mathcal{A}_{H,s}(\gamma)=\int_{[0,1]}( \gamma^*\theta - H\circ\gamma \, dt)+h(\gamma(0)) - h_s(\gamma(1)),
	\ee
	for $\gamma \in \mathcal{P}(L,N_{s})$. Here $h_s: N_s \to \R$ is the
	Liouville primitive of $N_s$. Recall that $N_s=\phi^s_G(\tilde{N}_1)$ and denote $S_s=h_s\circ\phi^s_G: \tilde{N}_1\to \mathbb{R}$. An elementary computation (see for example \cite[Proposition 3.4.8]{oh:book}) gives the following formula:
	\be\label{s_primitive}
	S_s=\tilde{h}+\int_0^s (\langle \theta,X_{G}\rangle - G)\circ \phi^t_G \  dt,
	\ee
	where $\tilde{h}$ is the Liouville primitive for $\tilde{N}_1$. Note that, by construction both $\tilde{N}_s$ and $N_s$ agree with cotangent fibers inside $D^*_R M$. Hence this formula shows that $\tilde{h}$ and $h_s$ are always zero inside $D^*_R M$. Therefore, since $N_1=T^*_{q_1}M$, $h_1$ must be zero by uniqueness up to constant of the Liouville primitive.
	In this way, for $x_1\in \CX(L,T_{q_1}^*M)$ we have $\mathcal{A}_{H,1}(x_1)=\mathcal{A}_{H}(x_1)$ and for $x_0\in \CX(L,N_0)$ we have $\mathcal{A}_{H,0}(x_0)=\mathcal{A}_{H}(x_0)$. Therefore for $u \in\mathcal{N}(x_0,x_1)$ we have
	$$
	\CA_{H} (x_1) - \CA_{H} (x_0)
	= \int_{-\infty}^\infty \frac{d}{d\tau}\left(\CA_{H;\rho(\tau)}(u(\tau))\right) \, d\tau.
	$$
	We compute
	\beastar
	\frac{d}{d\tau} (\CA_{H;\rho(\tau)}(u(\tau))) & = &
	\frac{d}{d\tau}\left(\int_{[0,1]} (u(\tau))^*\theta - H(u(\tau,t))\, dt + h(u(\tau,0)) \right)\\
	&{}& - \frac{d}{d \tau}(h_{\rho(\tau)}(u(\tau,1))).
	\eeastar
	For the first summand, we use the first variation with free boundary
	condition (see Equation (2.17) \cite{oh:jdg} for example) together with the fact that $u$ is a solution of \eqref{eq:movingCR}
	and get
	$$
	- \int_{[0,1]} \Big|\dudt - X_H(u)\Big|^2_{J_t}\,dt
	+  \left\langle \theta, \frac{\del u}{\del \tau}(\tau,1) \right\rangle.
	$$
	For the second, we start by noticing that the boundary condition $u(\tau,1)\in N_{\rho(\tau)}$ implies that $u(\tau,1)=\phi^{\rho(\tau)}_G(v(\tau))$ for some curve $v(\tau)\in \tilde{N}_1$. Hence $h_{\rho(\tau)}(u(\tau,1))=S_{\rho(\tau)}(v(\tau))$ and using \ref{s_primitive} we compute
	\begin{align}
	\frac{d}{d\tau} (h_{\rho(\tau)}&(u(\tau,1)))  =  d S_{\rho(\tau)}\left(\frac{d v}{d \tau}\right) + \rho'(\tau)\left(\frac{d S_s}{d s}\right)\bigg|_{s=\rho(\tau)}(v(\tau)) \\
	& = d h_{\rho(\tau)}\circ d \phi^{\rho(\tau)}_G\left(\frac{d v}{d \tau}\right) + \rho'(\tau)\left( \langle \theta, X_G \rangle(\phi^{\rho(\tau)}_G(v(\tau))) - G(\phi^{\rho(\tau)}_G(v(\tau)))\right).\nonumber
	\end{align}
	It follows from the definitions that
	$$\frac{\del u}{\del \tau}(\tau,1)= d \phi_G^{\rho(\tau)}\left(\frac{d v}{d \tau}\right) + \rho'(\tau) X_G(u(\tau,1)).$$
	Plugging this in the previous equation and using the definition of Liouville primitive we obtain
	\beastar
	\frac{d}{d\tau} (h_{\rho(\tau)}(u(\tau,1))) & = &  d h_{\rho(\tau)}\left(\frac{\del u}{\del \tau}(\tau,1)\right) - \rho'(\tau)d h_{\rho(\tau)}\left(X_G(u(\tau,1))\right)\nonumber\\
	& & + \rho'(\tau)\langle \theta, X_G \rangle(u(\tau,1)) - \rho'(\tau)G(u(\tau,1))\nonumber\\
	& = & \left\langle \theta, \frac{\del u}{\del \tau}(\tau,1) \right\rangle - \rho'(\tau)G(u(\tau,1)).
	\eeastar
	Combining the above calculations we conclude
	\be
	\frac{d}{d\tau} (\CA_{H;\rho(\tau)}(u(\tau)))= - \int_{[0,1]} \Big|\dudt - X_H(u)\Big|^2_{J_t}\,dt + \rho'(\tau)G(u(\tau,1)),\nonumber
	\ee
	which immediately proves the lemma.
\end{proof}

We have now finished the proof of Proposition \ref{prop1 GS}.

\end{document}